\documentclass[12pt, reqno]{amsart}
\usepackage{amsmath, amsthm, amscd, amsfonts, amssymb, graphicx, color}
\usepackage[bookmarksnumbered, colorlinks, plainpages]{hyperref}

\newtheorem{theorem}{Theorem}[section]
\newtheorem{lemma}[theorem]{Lemma}

\newtheorem{corollary}[theorem]{Corollary}
\theoremstyle{definition}

\theoremstyle{remark}

\numberwithin{equation}{section}

\newcommand{\N}{\ensuremath{{\mathbb N}}}

\newcommand{\R}{\ensuremath{{\mathbb R}}}

\newcommand{\AveP}{\underset{\pi}{\mbox{Ave}}}
\newcommand{\AvePP}{\underset{\pi,\sigma}{\mbox{Ave}}}

\newcommand{\abs}[1]{\left\lvert#1 \right\rvert}
\newcommand{\norm}[1]{\left \lVert#1 \right\rVert}

\begin{document}
\setcounter{page}{1}

\title[Musielak-Orlicz Spaces Isomorphic to Subspaces of $L_1$]{Musielak-Orlicz Spaces that are Isomorphic to Subspaces of $L_1$}

\author[Joscha Prochno]{Joscha Prochno}

\address{Institute of Analysis, Johannes Kepler University Linz, Altenbergerstrasse 69, 4040 Linz, Austria.}
\email{\textcolor[rgb]{0.00,0.00,0.84}{joscha.prochno@jku.at}}




\date{Received: xxxxxx; Revised: yyyyyy; Accepted: zzzzzz.
\newline \indent $^{*}$ Corresponding author}

\begin{abstract}
We prove that $\frac{1}{n!}\sum_{\pi\in\mathfrak{S}_n} \left( \sum\limits_{i=1}^n\abs{x_ia_{i,\pi(i)}}^2 \right)^{\frac{1}{2}}$ is equivalent to a Musielak-Orlicz norm $\norm{x}_{\Sigma M_i}$. We also obtain the inverse result, i.e., given the Orlicz functions, we provide a formula for the choice of the matrix that generates the corresponding Musielak-Orlicz norm. As a consequence, we obtain the embedding of 2-concave Musielak-Orlicz spaces into $L_1$.
\textbf{still under review by the journal}.
\end{abstract} \maketitle

\section{Introduction}

The variety of subspaces of $L_1$ is very rich and over the years, there was put tremendous effort in characterizing them. In \cite{key-BDC}, using the theorem of de Finetti, Bretagnolle and Dacunha-Castelle proved that an Orlicz space $\ell_M$ is isomorphic to a subspace of $L_1$ if and only if $M$ is equivalent to a 2-concave Orlicz function. The corresponding finite-dimensional version was proved in \cite{key-KS1} and \cite{key-S1} by Kwapie\'n and Sch\"utt, using combinatorial and probabilistic tools. To be more precise, in \cite{key-S1}, combined with the main result from \cite{key-KS1}, the author first proved that an Orlicz function $M$ has to be equivalent to a 2-concave Orlicz function if the corresponding Orlicz space $\ell_M^n$ is isomorphic to a subspace of $L_1$, and also obtained the inverse result, i.e., $\ell_M^n$ is isomorphic to a subspace of $L_1$ if $M$ is 2-concave.

Following the ideas of \cite{key-S1} and using results obtained in \cite{key-P}, we extend the first result from \cite{key-S1} to the case of Musielak-Orlicz spaces (definitions are given below). These generalized Orlicz spaces are defined using a different Orlicz function in each component. The first main result which we will prove in Section 3 is the following:

\begin{theorem} \label{THM hauptsatz 1}
  Let $(a_{i,j})_{i,j=1}^n \in \R^{n\times n}$ such that $a_{i,1} \geq \ldots \geq a_{i,n}>0$ for any $i=1,\ldots,n$. Let $M_1,\ldots,M_n$ be Orlicz functions so that for the conjugate functions and all $\ell=1,\ldots,n$
    \begin{equation}\label{EQU M ueber matrix definiert}
      M_i^{*-1}\left( \frac{\ell}{n} \right) = \left\{ \left( \frac{1}{n} \sum_{j=1}^{\ell} a_{i,j}\right)^2 + \frac{\ell}{n} \left( \frac{1}{n} \sum_{j=\ell+1}^n \abs{a_{i,j}}^2 \right) \right\}^{\frac{1}{2}},
    \end{equation}
  and where each $M_i^{*}$ is affine on the intervals $[\frac{\ell-1}{n},\frac{\ell}{n}]$ and extended linearly.
  Then, for all $x\in\R^n$,
    \begin{equation}\label{EQU l2 average}
      c_1 \norm{x}_{\Sigma M_i} \leq \AveP \left( \sum\limits_{i=1}^n\abs{x_ia_{i,\pi(i)}}^2 \right)^{\frac{1}{2}} \leq c_2 \norm{x}_{\Sigma M_i},
    \end{equation}
  where $c_1,c_2$ are positive absolute constants.    
\end{theorem}

In fact, as shown in \cite{key-S1} there is always an Orlicz function satisfying condition (\ref{EQU M ueber matrix definiert}).

Given strictly 2-concave Orlicz functions $M_1,\ldots, M_n$, in Section 4 we show how to choose the matrix $(a_{i,j})_{i,j}\in\R^{n\times n}$ so that (\ref{EQU l2 average}) is equivalent to the Musielak-Orlicz norm $\norm{\cdot}_{\Sigma M_i}$, where we closely follow the ideas of \cite{key-S1}.
The second main result is the following:

\begin{theorem}\label{THM hauptsatz 2}
  Let $M_1,\ldots,M_n$ be strictly convex, twice differentiable and strictly 2-concave Orlicz functions so that $M^{*}_i(1)=1$ for all $i=1,\ldots,n$. For all $i,j=1,\ldots,n$ let
    \begin{eqnarray}\label{EQU wahl der matrix}
      a_{i,j} & = & -\frac{n}{2} \int_{\frac{j-1}{n}}^{\frac{j}{n}} \left\{ \int_t^1 \frac{ \left( \left( M_i^{*-1} \right)^2 \right)''(s) }{ \sqrt{\left(M_i^{*-1}\right)^2(s) - s\left( \left( M_i^{*-1} \right)^2 \right)'(s) } } ds  \right. \cr
      && \left. +1-\sqrt{1-\left( \left( M_i^{*-1} \right)^2 \right)'(1)} \right\}.
    \end{eqnarray}
  Then, for all $x\in\R^n$,
    $$
      c_1 \norm{x}_{\Sigma M_i} \leq \AveP \left( \sum\limits_{i=1}^n\abs{x_ia_{i,\pi(i)}}^2 \right)^{\frac{1}{2}} \leq c_2 \norm{x}_{\Sigma M_i},
    $$
  where $c_1,c_2$ are positive absolute constants.     
\end{theorem}

The new idea here is to consider averages over matrices instead of just taking the average over a vector (see also \cite{key-P}). This corresponds to the idea of considering random variables which are not necessarily identically distributed. In fact, using this idea the results obtained in \cite{key-GLSW} can also be extended to the case of Musielak-Orlicz spaces.\\
We would also like to emphasize that, although the proofs are quite similar to the ones in \cite{key-S1}, the results we obtain provide important and crucial tools to find more general classes of subspaces of $L_1$, e.g., Orlicz-Lorentz spaces, Musielak-Orlicz-Lorentz spaces. Hence, it is seems absolutely essential to have them at hand. Moreover, it seems that these general classes of spaces cannot be obtained from the results
in \cite{key-KS1}, \cite{key-KS2} or \cite{key-S1} and, therefore, an extension of the combinatorial results to the more general setting is crucial to find easily applicable characterizations to decide whether a given Banach space is isomorphic to a subspace of $L_1$.
 
Additionally, in view of the combinatorial approach we use in this work, we would like to point out that combinatorial inequalities, first studied in \cite{key-KS1}, \cite{key-KS2} and later extended in \cite{key-PS} and \cite{key-S2}, turned out to be very fruitful to characterize subspaces of $L_1$. For instance, in \cite{key-PS} the authors recently obtained combinatorial results similar to Theorem \ref{THM hauptsatz 1} and gave an easily applicable characterization for products of Orlicz spaces, i.e., spaces of the form $\ell^n_M(\ell^n_N)$, to embed into $L_1$. Hence, the combinatorial inequalities are interesting in themselves. For further applications of those combinatorial methods see \cite{key-MSS}, \cite{key-P}, \cite{key-RS} , \cite{key-S2} or \cite{key-S3}.

\section{Preliminaries}
A convex function $M:[0,\infty) \to [0,\infty)$ with $M(0)=0$ and $M(t)>0$ for $t>0$ is called an Orlicz function. We say that $M$ is 2-concave if $M\circ\sqrt{\cdot}$ is a concave function. The n-dimensional Orlicz space $\ell^n_M$ is $\R^n$ equipped with the norm
  $$
    \norm{x}_M = \inf\left\{ \rho>0 : \sum_{i=1}^n M\left(\frac{\abs{x_i}}{\rho}\right) \leq 1  \right\}.
  $$
Given an Orlicz function $M$ we define its conjugate function $M^*$ via the Legendre-Transform
  $$
    M^*(x) = \sup_{t\in[0,\infty)}(xt-M(t)).
  $$
For instance, taking $M(t)=\frac{1}{p}t^p$, $p\geq 1$, the conjugate function is given by $M^*(t)=\frac{1}{p^*}t^{p^*}$ with $\frac{1}{p^*}+\frac{1}{p}=1$.  A more detailed and thorough introduction to Orlicz spaces 
can be found in \cite{key-KR} and \cite{key-RR}.\\
Let $M_1,\ldots,M_n$ be Orlicz functions. We define the Musielak-Orlicz space $\ell^n_{\Sigma M_i}$ to be $\R^n$ equipped with the norm
  $$
    \norm{x}_{\Sigma M_i} = \inf \left\{ \rho>0 : \sum_{i=1}^n M_i\left( \frac{\abs{x_i}}{\rho} \right) \leq 1 \right\}.
  $$
These spaces may be considered as generalized Orlicz spaces. In fact, one can easily show, using Young's inequality, that the norm of the dual space $(\ell^n_{\Sigma M_i})^*$ is equivalent to
  $$
    \norm{x}_{\Sigma M_i^*} = \inf \left\{ \rho>0 \,:\, \sum_{i=1}^n M_i^*\left( \frac{\abs{x_i}}{\rho} \right) \leq 1 \right\},
  $$
which is the analog result as for the classical Orlicz spaces.\\
We will use the notation $a\sim b$ to express that there exist two positive absolute constants $c_1, c_2$ such that $c_1a\leq b\leq c_2 a$. Similarly, we write $a \lesssim b$ if there exists a positive absolute constant $c$ such that $a \leq c b$. The letters 
$c,C,c_1,c_2,\ldots$ will denote positive absolute constants, whose value may change from line to line. \\
Furthermore, for a non-empty set $I \subset \{1,\ldots,m\}$, $m\in \N$, we write $\abs{I}$ for the cardinality of $I$.\\ 
Given a matrix $a=(a_{i,j})_{i,j=1}^n\in\R^{n\times n}$ we denote the decreasing rearrangement of $\abs{a_{i,j}}$, $i,j=1,\ldots,n$ by $s(1),\ldots,s(n^2)$.\\
We say that two Orlicz functions $M$ and $N$ are equivalent if there are positive absolute constants  $c_1$ and $c_2$ such that for all $t \geq 0$
  $$
     c_1 N^{-1}(t) \leq M^{-1}(t) \leq c_2 N^{-1}(t).
  $$
In this case we write $N\sim M$. If two Orlicz functions are equivalent so are their norms. We say that two sequences of Orlicz functions $(M_i)_{i=1}^n$, $(N_i)_{i=1}^n$ are uniformly equivalent, 
if there are positive absolute constants  $c_1$ and $c_2$ such that for all $t \geq 0$ and all $i=1,\ldots,n$  
$$
  c_1 N_i^{-1}(t) \leq M_i^{-1}(t) \leq c_2 N_i^{-1}(t).
$$
In this case, the corresponding Musielak-Orlicz norms are equivalent. \\
In the following, $\pi,\sigma$ are permutations of $\{1,\ldots,n\}$ and we write $\AveP$ to denote the average over all permutations in the group 
$\mathfrak{S}_n$, {\it i.e.}, $\AveP:=\frac{1}{n!}\sum_{\pi\in\mathfrak{S}_n}$.\\
Moreover, we define the space $L_1^{2^nn!}$ to be
$$
  L_1^{2^nn!} = \{ (x(\epsilon,\pi))_{\epsilon,\pi} \,:\, \epsilon_i=\pm 1, i=1,\ldots,n, \pi \in \mathfrak{S}_n \},
$$
equipped with the norm
$$
  \norm{x}_{L_1^{2^nn!}} = \frac{1}{2^nn!} \sum_{\epsilon,\pi} \abs{x(\epsilon,\pi)}.
$$
The Banach-Mazur distance of two Banach spaces $X$ and $Y$ is defined by
    $$
      d(X,Y) = \inf\left\{ \norm{T}\norm{T^{-1}} \,:\, ~ T\in L(X,Y) ~ \hbox{isomorphism} \right\}.
    $$
Let $(X_n)_n$ be a sequence of $n$-dimensional normed spaces and let $Z$ also be a normed space. If there exists a constant $C>0$ such that for all $n\in\N$ there exists a normed space $Y_n \leq Z$ with $\dim(Y_n)=n$ and $d(X_n,Y_n)\leq C$, then we say that $(X_n)_n$ embeds uniformly into $Z$ or in short: $X_n$ embeds into $Z$. For a detailed introduction to the concept of Banach-Mazur distances, see for example \cite{key-TJ}.  

The following result was obtained in \cite{key-KS2}:
\begin{lemma}\label{LEM 3-dimensionale matrix ueber n mal n groesste}
  For all $n\in N$ and all numbers $a(i,j,k)$, $i,j,k=1,\ldots,n$
    $$
        \AvePP \max_{1\leq i \leq n} \abs{a(i,\pi(i),\sigma(i))} \sim \frac{1}{n^2} \sum_{k=1}^{n^2} s(k),
    $$
  where $s(1),\ldots,s(n^3)$ is the decreasing rearrangement (d.r.a.) of the numbers $\abs{a(i,j,k)}$, $i,j,k=1,\ldots,n$.   
\end{lemma}

Now, let $n,N\in\N$ with $n\leq N$. For a matrix $a\in\R^{n\times N}$ with $a_{i1} \geq \ldots \geq a_{i,N}>0$, $i=1,\ldots,n$, 
we define a norm on $\R^n$ by
  $$
    \norm{x}_a = \max_{\sum_{i=1}^n \ell_i\leq N} \sum_{i=1}^n \left( \sum_{j=1}^{\ell_i}a_{i,j} \right) \abs{x_i},~x\in\R^n.
  $$
 
The next lemma is a generalization of a result from \cite{key-KS2} and was recently obtained in \cite{key-P}:
\begin{lemma}\label{LEM matrixnorm}
  Let $n,N\in\N$ with $n\leq N$. Let $a\in\R^{n\times N}$ so that $a_{i,1}\geq \ldots \geq a_{i,N}>0$
  for all $i=1,\ldots,n$. Let $M_1,\ldots, M_n$ be Orlicz functions such that for the conjugate functions $M_1^{*}, \ldots, M_n^{*}$ and all $m=1,\ldots,N$
    $$
      M_i^{*}\left( \sum_{j=1}^{m}a_{i,j} \right)=\frac{m}{N}.
    $$
  Then, for all $x\in\R^n$,
    $$
      \frac{1}{2}\norm{x}_a \leq \norm{x}_{\Sigma M_i} \leq 2 \norm{x}_a.
    $$  
\end{lemma}

This is a very useful result to estimate Orlicz norms (in case all rows of $a$ are the same) respectively Musielak-Orlicz norms in general.

\section{Generating Musielak-Orlicz Spaces}

It was shown in \cite{key-P} that an $\ell_{\infty}$-average over permutations, {\it i.e.}, \\ $\AveP \max\limits_{1\leq i \leq n} |x_ia_{i,\pi(i)}|$ is equivalent to a Musielak-Orlicz norm depending on the matrix $a\in\R^{n\times n}$. 
In proving the main theorem, we will extend this result to the more sophisticated case of an $\ell_2$-average. This is crucial because this $\ell_2$-average in (\ref{EQU l2 average}) is
equivalent to the $L_1$-norm and, therefore, gives rise to a subspace of $L_1$. Following the ideas of \cite{key-S1}, we will now prove the main theorem. Though the proof is quite similar to the one in \cite{key-S1}, for the sake of completeness and because it is easier to read, we include the details here.

\begin{proof}(Theorem \ref{THM hauptsatz 1})
Let $a=(a_{i,j})_{i,j=1}^n\in\R^{n\times n}$ with $a_{i,1} \geq \ldots \geq a_{i,n}$ for any $i=1,\ldots n$. For $k=1,\ldots,n$ we define $b_k = \sqrt{\frac{n}{k}}$. 
It follows from Lemma 2.5 in \cite{key-KS2} (or Theorem 3.4 in \cite{key-P} in the more general case) that this vector generates the $\ell_2$ norm, {\it i.e.}, for any $y\in\R^n$
    $$
        \underset{\sigma}{\mbox{Ave}} \max_{1\leq k \leq n}\abs{y_kb_{\sigma(k)}} \sim \norm{y}_2.
    $$
Therefore, for all $x\in\R^n$, 
  $$
    \AvePP \max_{1\leq i \leq n} \abs{ x_i a_{i,\pi(i)} b_{\sigma(i)} } \sim \AveP \left( \sum_{i=1}^n \abs{x_i a_{i,\pi(i)}}^2 \right)^{\frac{1}{2}}.
  $$
We apply Lemma \ref{LEM 3-dimensionale matrix ueber n mal n groesste} to the 3-dimensional matrix $(x_i a_{i,j} b_k)_{i,j,k=1}^n$ and obtain
  $$
    \AvePP \max_{1\leq i \leq n} \abs{ x_i a_{i,\pi(i)} b_{\sigma(i)} } \sim \frac{1}{n^2} \sum_{\ell=1}^{n^2} s(\ell),
  $$
where $s(1),\ldots,s(n^3)$ is the d.r.a. of $\abs{x_i a_{i,j} b_k}$, $i,j,k=1,\ldots,n$. Thus
  $$
    \AveP \left( \sum_{i=1}^n \abs{x_i a_{i,\pi(i)}}^2 \right)^{\frac{1}{2}} \sim \frac{1}{n^2} \sum_{\ell=1}^{n^2} s(\ell).
  $$
The latter expression can be written in the following way:
  $$
    \frac{1}{n^2} \sum_{\ell=1}^{n^2} s(\ell) = \frac{1}{n^2} \sum_{i=1}^n \abs{x_i} \sum_{(j,k)\in I_i} \abs{a_{i,j} b_k},
  $$
with $\sum_{i=1}^n \abs{I_i} = n^2$ and $I_1,\ldots,I_n$ chosen so that the upper sum is maximal, i.e.,
  $$
    \sum_{(j,k)\in I_i} \abs{a_{i,j} b_k} = \sum_{\ell=1}^{\abs{I_i}} s(\ell).
  $$
Lemma \ref{LEM matrixnorm} yields
  $$
    \AveP \left( \sum_{i=1}^n \abs{x_i a_{i,\pi(i)}}^2 \right)^{\frac{1}{2}} \sim \norm{x}_{\Sigma N_i},
  $$
where the functions $N_1,\ldots,N_n$ are defined for all $\ell=1,\ldots,n$ by
  $$
    N_i^{*}\left( \frac{1}{n^2} \sum_{j=1}^{\ell} t^i(j) \right) = \frac{\ell}{n^2},
  $$
are affine between the given values, extended linearly, and for each $i=1,\ldots,n$ the sequence $t^i(1),\ldots,t^i(n^2)$ is the d.r.a. of $\abs{a_{ij}b_k}$, $j,k=1,\ldots,n$.
We have to show that $\norm{\cdot}_{\Sigma N_i} \sim \norm{\cdot}_{\Sigma M_i}$. \\ 
Again, by a change of variable, for all $i=1,\ldots,n$ we can write
  $$
    N_i^{*-1}\left( \frac{\ell}{n} \right) = \frac{1}{n^2} \sum_{j=1}^{\ell n} t^i(j) = \frac{1}{n^2} \sum_{j=1}^n a_{i,j} \sum_{k=1}^{m_j^i} b_k,
  $$ 
where $m_j^i \leq n$, $\sum_{j=1}^n m_j^i = \ell n$ and $m_1^i,\ldots,m_n^i$ maximize the upper sum. Since we chose $b_k=\sqrt{\frac{n}{k}}$, $k=1,\ldots,n$, by approximation we obtain 
for all $i=1,\ldots,n$
  $$
    N_i^{*-1}\left( \frac{\ell}{n} \right) = \frac{1}{n^2} \sum_{j=1}^n a_{i,j} \sum_{k=1}^{m_j^i} b_k \leq \frac{2}{n^{\frac{3}{2}}} \sum_{j=1}^n a_{i,j} \sqrt{m_j^i}.
  $$
Notice that $m_1^i \geq \ldots \geq m_n^i$ for any $i=1,\ldots,n$, since we assumed $a_{i,1} \geq \ldots \geq a_{i,n}$ for all $i=1,\ldots,n$. Hence, by H\"older's inequality, for all $i=1,\ldots,n$ and all 
$\ell=1,\ldots,n$ 
  \begin{eqnarray*}
    N_i^{*-1}\left( \frac{\ell}{n} \right) & \leq & \frac{2}{n^{\frac{3}{2}}} \left\{ \sqrt{m_1^i} \sum_{j=1}^{\ell} a_{i,j} + \sum_{j=\ell+1}^n a_{i,j} \sqrt{m_j^i} \right\} \\
    & = &  \frac{2}{n^{\frac{3}{2}}} \norm{(\sqrt{m_1^i}\sum_{j=1}^{\ell} a_{i,j}, \sqrt{m^i_{\ell+1}} a_{i,\ell+1}, \ldots, \sqrt{m_n^i} a_{i,n} )}_1 \\
    & \leq & \frac{2}{n^{\frac{3}{2}}} \norm{\left( \sum_{j=1}^{\ell} a_{i,j}, \sqrt{\ell} a_{i,\ell+1}, \ldots, \sqrt{\ell} a_{i,n} \right)}_2 \times \cr
    && \norm{(\sqrt{m_1^i}, \sqrt{m_{\ell+1}^i}\frac{1}{\sqrt{\ell}}, \ldots, \sqrt{m_{n}^i}\frac{1}{\sqrt{\ell}} )}_2 \\
    & = & \frac{2}{n^{\frac{3}{2}}} \left( \left( \sum_{j=1}^{\ell} a_{i,j} \right)^2 + \ell \sum_{j=\ell+1}^n \abs{a_{i,j}}^2 \right)^{\frac{1}{2}} \left(m_1^i + \frac{1}{\ell} \sum_{j=\ell +1}^n m_j^i \right)^{\frac{1}{2}}.
  \end{eqnarray*}
Since $\sum_{j=1}^n m_j^i = \ell n$ for all $i=1,\ldots,n$, we obtain
  $$
    \left(m_1^i + \frac{1}{\ell} \sum_{j=\ell +1}^n m_j^i \right)^{\frac{1}{2}} \leq \sqrt{n},
  $$
and, therefore, for all $i=1,\ldots,n$
  $$
    N_i^{*-1}\left( \frac{\ell}{n} \right)  \leq  \frac{2}{n} \left( \left( \sum_{j=1}^{\ell} a_{i,j} \right)^2 + \ell \sum_{j=\ell+1}^n \abs{a_{i,j}}^2 \right)^{\frac{1}{2}}
  $$
Hence, for all $i=1,\ldots,n$
  $$
    N_i^{*-1}\left( \frac{\ell}{n} \right)  \leq  2 \left( \left( \frac{1}{n} \sum_{j=1}^{\ell} a_{i,j} \right)^2 + \frac{\ell}{n}\left( \frac{1}{n} \sum_{j=\ell+1}^n \abs{a_{i,j}}^2\right) \right)^{\frac{1}{2}}
    = 2 M_i^{*-1}\left( \frac{\ell}{n} \right).
  $$
Now we prove the lower estimate. For all $i=1,\ldots,n$ and all $\ell=1,\ldots,n$
  $$
    N_i^{*-1}\left( \frac{\ell}{n} \right) = \frac{1}{n^2} \sum_{j=1}^{\ell n} t^i(j)  \geq \frac{1}{n} \sum_{j=1}^{\ell} a_{i,j},
  $$
where we chose $m_j^i = n$ for $j=1,\ldots, \ell$ and $m_j^i=0$ for $j>\ell+1$. As in \cite{key-KS2}, we have for all $i=1,\ldots,n$ and all $\ell=1,\ldots,n$
  $$
    \frac{1}{n} \sum_{j=1}^{\ell} a_{i,j} \geq \frac{\sqrt{\ell}}{2n} \left( \sum_{j=1}^{\ell} \abs{a_{i,j}}^2 \right)^{\frac{1}{2}},
  $$
i.e., we get
  $$
    2 N_i^{*-1}\left( \frac{\ell}{n} \right) \geq \frac{\sqrt{\ell}}{n} \left( \sum_{j=1}^{\ell} \abs{a_{i,j}}^2 \right)^{\frac{1}{2}}.
  $$
Hence, 
  \begin{eqnarray*}
    3 N_i^{*-1}\left( \frac{\ell}{n} \right) & \geq & \frac{1}{n} \sum_{j=1}^{\ell} a_{i,j} + \frac{\sqrt{\ell}}{n} \left( \sum_{j=1}^{\ell} \abs{a_{i,j}}^2 \right)^{\frac{1}{2}} \\
    & \geq & \left\{ \left( \frac{1}{n} \sum_{j=1}^{\ell} a_{i,j} \right)^2 + \frac{\ell}{n} \left( \frac{1}{n} \sum_{j=1}^{\ell} \abs{a_{i,j}}^2 \right) \right\}^{\frac{1}{2}} \\
    & = & M_i^{*-1}\left( \frac{\ell}{n} \right).
  \end{eqnarray*}   
Thus, for all $i=1,\ldots,n$ and all $\ell=1,\ldots,n$,
  $$
    \frac{1}{2} N_i^{*-1}\left( \frac{\ell}{n} \right)  \leq M_i^{*-1} \leq 3 N_i^{*-1}\left( \frac{\ell}{n} \right),
  $$
and since the functions are affine on the intervals $[\frac{\ell-1}{n}, \frac{\ell}{n}]$ and then extended linearly, this establishes the upper inequalities for all values and hence
  $$
    \frac{1}{2} \norm{x}_{\Sigma N_i} \leq \norm{x}_{\Sigma M_i} \leq 3 \norm{x}_{\Sigma N_i}.
  $$ 
Therefore, there exist positive absolute constants $c_1,c_2$ such that
  $$
    c_1 \norm{x}_{\Sigma M_i} \leq \AveP \left( \sum\limits_{i=1}^n\abs{x_ia_{i,\pi(i)}}^2 \right)^{\frac{1}{2}} \leq c_2 \norm{x}_{\Sigma M_i}.
  $$                            
\end{proof}

\section{The Embedding of $\ell^n_{\Sigma M_i}$ into $L_1$}

Before we can prove Theorem \ref{THM hauptsatz 2}, which is crucial to obtain the embedding of 2-concave Musielak-Orlicz spaces into $L_1$, we need some technical results.
The following lemma was obtained in \cite{key-S1} but we just state the part of the lemma we need:

\begin{lemma}\label{LEM carsten}
  Let $H:[0,1]\to \R$ be a concave, increasing and twice continuously differentiable function on $(0,1]$. Let $H(0)=0$ and assume that $\left(\frac{H(t)}{t}\right)' \neq 0$ for all $t\in(0,1]$.
  Then
  \begin{itemize}
   \item[(i)] The function $$f(t)= \frac{1}{2} \int_t^1 \frac{H''(s)}{\sqrt{H(s)-sH'(s)}} ds + \sqrt{H(1)} - \sqrt{H(1)-H'(1)}$$ is well defined, non-negative, decreasing and differentiable on $(0,1]$.
   \item[(ii)] The integral $\int_0^1 f(t) dt$ is finite and for all $t\in[0,1]$ $$H(t) = \left( \int_0^t f(s) ds \right)^2 + t \int_t^1 \abs{f(s)}^2 ds.$$
  \end{itemize}
\end{lemma}

Furthermore, it is easy to observe that $M^{'-1}$ exists for any strictly convex Orlicz function $M$ and that, additionally, $M^{*'}(t) = M^{'-1}(t)$. If the Orlicz function is twice differentiable 
then so is $M^{*-1}$. Moreover, if $M$ is strictly 2-concave then $(M^{*-1})^2$ is strictly concave, and hence
  $$
    0 > \left( (M^{*-1})^2 \right)'(s) - \frac{(M^{*-1})^2(s)}{s} = s\frac{d}{ds} \left( \frac{(M^{*-1})^2(s)}{s} \right).
  $$  
So we can apply Lemma \ref{LEM carsten} to the functions $H_i(t) = (M_i^{*-1})^2(t)$, $i=1,\ldots,n$, which we will do to prove the theorem. 

\begin{proof}(Theorem \ref{THM hauptsatz 2})
  For each $i=1,\ldots,n$ we choose $H_i(t) = (M_i^{*-1})^2(t)$ which is well defined (see remarks above). Part (i) of Lemma \ref{LEM carsten} yields that for any $i=1,\ldots,n$ the sequence
  $a_{i,1},\ldots,a_{i,n}$ given by (\ref{EQU wahl der matrix}), is positive and decreasing. As in Lemma \ref{LEM carsten}, we define for every $i=1,\ldots,n$
    $$
      f_i(t)= \frac{1}{2} \int_t^1 \frac{H_i''(s)}{\sqrt{H_i(s)-sH_i'(s)}} ds + \sqrt{H_i(1)} - \sqrt{H_i(1)-H_i'(1)}. 
    $$ 
  Then, in terms of those functions,  
    $$
      a_{i,j} = n \int_{\frac{j-1}{n}}^{\frac{j}{n}} f_i(s) ds,~i,j=1,\ldots,n.
    $$ 
  Part (ii) of Lemma \ref{LEM carsten} gives
    $$
      (M_i^{*-1})^2(t) = \left( \int_0^t f_i(s) ds \right)^2 + t \int_t^1 \abs{f_i(s)}^2 ds,
    $$
  and hence, for all $j=1,\ldots,n$
    \begin{eqnarray*}
      M_i^{*-1} \left( \frac{j}{n} \right) & = & \left( \left( \int_0^{\frac{j}{n}} f_i(s) ds \right)^2 + \frac{j}{n} \int_{\frac{j}{n}}^1 \abs{f_i(s)}^2 ds \right)^{\frac{1}{2}} \\
      & = & \left( \left( \frac{1}{n}\sum_{k=1}^j a_{i,k} \right)^2 + \frac{j}{n} \left( \sum_{k=j}^{n-1} \int_{\frac{k}{n}}^{\frac{k+1}{n}} \abs{f_i(s)}^2 ds \right) \right)^{\frac{1}{2}}, \\ 
    \end{eqnarray*}
  since $\int_0^{j/n}f_i(s)ds = \sum_{k=1}^j\frac{a_{i,k}}{n}$. Because for all $i=1,\ldots,n$ the functions $f_i$ are non-negative and decreasing, we obtain for all $i,j=1,\ldots,n$
    $$
      a_{i,j} = n \int_{\frac{j-1}{n}}^{\frac{j}{n}} f_i(s) ds \leq n \int_{\frac{j-1}{n}}^{\frac{j}{n}} f_i\left(\frac{j-1}{n}\right) ds = f_i\left( \frac{j-1}{n} \right)
    $$
  and
    $$
      a_{i,j} = n \int_{\frac{j-1}{n}}^{\frac{j}{n}} f_i(s) ds \geq n \int_{\frac{j-1}{n}}^{\frac{j}{n}} f_i\left(\frac{j}{n}\right) ds = f_i\left( \frac{j}{n} \right).
    $$
  Thus, for all $i=1,\ldots,n$ and all $j=1,\ldots,n$
    $$
      M_i^{*-1}\left( \frac{j}{n} \right) \sim \left\{ \left( \frac{1}{n} \sum_{k=1}^{j} a_{i,k}\right)^2 + \frac{j}{n} \left( \frac{1}{n} \sum_{k=j+1}^n \abs{a_{i,k}}^2 \right) \right\}^{\frac{1}{2}}.
    $$
   We apply Theorem \ref{THM hauptsatz 1} and because the functions we obtain from this theorem are uniformly equivalent, we get that for all $x\in\R^n$
    $$
      c_1 \norm{x}_{\Sigma M_i} \leq \AveP \left( \sum\limits_{i=1}^n\abs{x_ia_{i,\pi(i)}}^2 \right)^{\frac{1}{2}} \leq c_2 \norm{x}_{\Sigma M_i},
    $$
  where $c_1,c_2$ are positive absolute constants.      
\end{proof}

The condition $M^{*}_i(1)=1$, $i=1,\ldots,n$ in the theorem is just a matter of normalization to assure that constants do not depend on the Orlicz functions. Moreover, assuming $M$ to be twice differentiable, strictly convex and strictly 2-concave can surely be omitted by approximation arguments. Strictly speaking, in the last part of the proof we would have to switch to equivalent Orlicz functions which we did not for reasons of simplicity.\\  
Now, from Theorem \ref{THM hauptsatz 2} we immediately obtain the following corollary:

\begin{corollary}
  Let $M_1,\ldots, M_n$ be strictly convex, twice differentiable and strictly 2-concave Orlicz functions. Then there exists a constant $C>1$ such that for every $n\in\N$, 
  there is a subspace $Y_n$ of $L_1^{2^nn!}$ with $dim(Y_n)=n$ and 
    $$
      d(\ell^n_{\Sigma M_i}, Y_n) \leq C.
    $$  
\end{corollary}
\begin{proof}
  We define the embedding as follows:
    $$
      \Psi_n : \ell_{\Sigma M_i}^n \to L_1^{2^nn!}, (x_i)_{i=1}^n \mapsto \left( \sum_{i=1}^n x_i \epsilon_i a_{i,\pi(i)} \right)_{\epsilon,\pi},
    $$
  where the matrix $(a_{i,j})_{i,j=1}^n \in\R^{n\times n}$ is chosen according to (\ref{EQU wahl der matrix}).  Then, for all $x\in\R^n$,
    $$
      \norm{\Psi_n(x)}_{L_1^{2^nn!}} = \frac{1}{2^nn!} \sum_{\epsilon,\pi} \abs{\sum_{i=1}^n x_i \epsilon_i a_{i,\pi(i)}},
    $$
  and using Khintchine's inequality we obtain
    $$
     \frac{1}{\sqrt{2}} \AveP \left( \sum\limits_{i=1}^n\abs{x_ia_{i,\pi(i)}}^2 \right)^{\frac{1}{2}} \leq \norm{\Psi_n(x)}_{L_1^{2^nn!}} \leq \AveP \left( \sum\limits_{i=1}^n\abs{x_ia_{i,\pi(i)}}^2 \right)^{\frac{1}{2}}.
    $$
  By Theorem \ref{THM hauptsatz 2}
    $$
      \frac{c}{\sqrt{2}} \norm{x}_{\Sigma M_i} \leq \norm{\Psi_n(x)}_{L_1^{2^nn!}} \leq C \norm{x}_{\Sigma M_i}.
    $$       
\end{proof}
Therefore, the corollary says that the sequence of spaces $\ell_{\Sigma M_i}^n$, $n\in \N$, embeds uniformly into $L_1$. Taking $M_i(\cdot)=M(\alpha_i \cdot)$, $i=1,\ldots,n$ for an Orlicz function $M$ satisfying the conditions above and a weight sequence 
$(\alpha_i)_{i=1}^n\in\R^n$, we obtain the embedding of weighted Orlicz spaces into $L_1$, for instance, weighted $\ell_p$ spaces where $M(t)=t^p$, $1<p<2$. 


\bibliographystyle{amsplain}

\end{document}